\documentclass[11pt]{article}
\usepackage{amsthm, amsmath, amssymb, amsfonts, url, booktabs, tikz, setspace, fancyhdr}
\usepackage[hidelinks]{hyperref}
\usepackage[margin = 1in]{geometry}
\usepackage{hyperref, enumerate}
\usetikzlibrary{arrows.meta}

\usepackage{enumitem}
\newlist{steps}{enumerate}{1}
\setlist[steps, 1]{label = Step \arabic*:}

\newtheorem{theorem}{Theorem}[section]

\newtheorem{lemma}[theorem]{Lemma}
\newtheorem{corollary}[theorem]{Corollary}
\newtheorem{conjecture}[theorem]{Conjecture}

\newtheorem{question}[theorem]{Question}

\theoremstyle{definition}
\newtheorem{definition}[theorem]{Definition}

\theoremstyle{remark}

\newcommand{\EE}{\mathbb{E}}
\newcommand{\PP}{\mathbb{P}}

\newcommand{\HH}{\mathbb{H}}

\newcommand{\FF}{\mathcal{F}}
\newcommand{\TT}{\mathcal{T}}

\newcommand{\Hom}{\mathrm{Hom}}

\newcommand{\xx}{\mathbf{x}}
\newcommand{\X}{\mathbf{X}}
\newcommand{\Y}{\mathbf{Y}}

\usetikzlibrary{trees}
\tikzstyle{p}+=[fill=black, circle, minimum width = 1pt, inner sep =
1pt]
\tikzstyle{w}+=[fill=white, draw, circle, minimum width = 1pt, inner sep =
1.5pt]

\begin{document}

\title{On some graph densities in locally dense graphs}

\author{Joonkyung Lee\thanks{Fachbereich Mathematik, Universit\"at Hamburg, Germany.
E-mail: {\tt
joonkyung.lee@uni-hamburg.de}. Research supported by ERC Consolidator Grant PEPCo 724903.}}

\date{}

\maketitle

\begin{abstract}
The Kohayakawa--Nagle--R\"odl--Schacht conjecture roughly states that every sufficiently large locally $d$-dense graph $G$ on $n$ vertices must contain at least $(1-o(1))d^{|E(H)|}n^{|V(H)|}$ copies of a fixed graph $H$.
Despite its important connections to both quasirandomness and Ramsey theory, there are very few examples known to satisfy the conjecture.

We provide various new classes of graphs that satisfy the conjecture.
Firstly, we prove that adding an edge to a cycle or a tree produces graphs that satisfy the conjecture.
Secondly, we prove that a class of graphs obtained by gluing complete multipartite graphs in a tree-like way satisfies
the conjecture.
We also prove an analogous result with odd cycles replacing complete multipartite graphs.
\end{abstract}


\section{Introduction}
A \emph{graph homomorphism} is a vertex map from a graph $H$ to another graph $G$ that preserves adjacency,
and the \emph{graph homomorphism density} $t_H(G)$ is the probability that a random vertex map from $H$ to $G$ is a graph homomorphism, i.e.,
\begin{align*}
 t_H(G):=\frac{|\Hom(H,G)|}{|V(G)|^{|V(H)|}}.
\end{align*}
Many statements in extremal graph theory can be rephrased as inequalities 
between certain graph homomorphism densities,
especially when $H$ is a fixed graph and the target graph $G$ is large.
For example, we say that, for a constant $0<d<1$, 
a graph sequence $G_n$ with $|V(G_n)|\rightarrow\infty$ and $t_{K_2}(G)\rightarrow d$ is quasirandom if and only if
\begin{align}\label{eq:counting}
	t_{H}(G_n)=(1\pm o(1))d^{|E(H)|},
\end{align}
for every fixed graph $H$, that is, the $H$-count is random-like in $G$.
A fundamental observation in the theory of quasirandom graphs, 
due to Thomason \cite{Thom87} and Chung, Graham, and Wilson \cite{CGW89}, 
states that $G_n$ is quasirandom if and only if
every subset $X\subseteq V(G_n)$ spans
\begin{align}\label{eq:density}
\frac{d}{2}|X|^2\pm o(|V(G_n)|^2)
\end{align}
edges. That is, we have a uniform edge density 
everywhere up to an error dominated by $|V(G_n)|^2$.

It is then natural to ask if 
some modifications of \eqref{eq:counting} or \eqref{eq:density} imply
variations of the other.
For instance, we say that a graph $G$ is $(\rho,d)$-dense if every vertex subset $X\subseteq V(G)$ of size at least 
$\rho |V(G)|$ contains at least $\frac{d}{2}|X|^2$ edges.
This is a weaker condition than \eqref{eq:density},
as we do not have an upper bound for the number of edges spanned by a vertex subset $X$.
Thus, we cannot fully recover \eqref{eq:counting}, but it is still plausible
that we can recover the lower bound
\begin{align*}
t_H(G_n)\geq (1-o(1))d^{|E(H)|}
\end{align*}
when $\rho$ is sufficiently small. This question was in fact formalised by Kohayakawa, Nagle, R\"odl, and Schacht~\cite{KNRS10}.
\begin{conjecture}[\cite{KNRS10}]\label{conj:KNRS}
	Let $H$ be a graph
	and let $0<d<1$ be fixed. 
	Then for every $\eta>0$, there exists $\rho=\rho(\eta,d,H)>0$ such that 
		\begin{align}\label{eq:KNRS}
			t_{H}(G)\geq (1-\eta) d^{|E(H)|}
		\end{align}			
	for every sufficiently large $(\rho,d)$-dense graph $G$.
\end{conjecture}

This conjecture is not an arbitrary variant of graph quasirandomness, 
but has natural applications to Ramsey theory.
Indeed, the notion of $(\rho,d)$-dense graphs already appears in a paper of Graham, R\"{o}dl, and Rucinski \cite{GRR01},
where they use it to bound the Ramsey number of sparse graphs.
Roughly speaking, given a 2-edge-colouring of a complete graph,
one colour is very dense on some vertex 
subset or the other colour is somewhat dense on all vertex subsets.
In the first case, it is usually simple to embed a graph $H$, while in the second case,
the problem reduces to an embedding problem in locally dense graphs.
It is therefore of significant interest to understand when we can embed a graph $H$
in a locally dense graph and how many copies we obtain. 

\medskip

Conjecture \ref{conj:KNRS} is also closely related to another beautiful conjecture of Sidorenko \cite{Sid92} and Erd\H{o}s--Simonovits \cite{ESi83}.
\begin{conjecture}[Sidorenko's conjecture~\cite{ESi83,Sid92}]\label{conj:Sido}
	Let $H$ be a bipartite graph
	and let $G$ be a graph.
	Then 
		\begin{align}\label{eq:Sido}
			t_{H}(G)\geq t_{K_2}(G)^{|E(H)| }.
		\end{align}			
\end{conjecture}
We say that a bipartite graph has \emph{Sidorenko's property} if and only if \eqref{eq:Sido} holds for all graphs $G$.
There are a number of graphs known to have Sidorenko's property \cite{CFS10, CKLL15,CL16,CL18,H10,KLL14,LSz12,Sid93,Sid92,Sz15}, but the conjecture still remains open.

\medskip

In this paper, we focus on Conjecture \ref{conj:KNRS}.
It is straightforward to see that
Conjecture \ref{conj:KNRS} is true for 
$H$ having Sidorenko's property with $\rho=1$ and $d=t_{K_2}(G)$.
For example, an even cycle or a tree has Sidorenko's property, as was firstly proven by Sidorenko~\cite{Sid93}, and hence satisfies Conjecture~\ref{conj:KNRS}. 
As a partial converse, it is shown in \cite{CKLL15} that the 1-subdivision of every $H$ that satisfies Conjecture~\ref{conj:KNRS} has Sidorenko's property.

On the other hand, there are only very few non-bipartite graphs known to satisfy the conjecture. It is folklore that 
the complete $\ell$-partite graph 
$K(r_1,r_2,\cdots,r_\ell)$ on $r=r_1+\cdots+r_\ell$ vertices
satisfies the conjecture,\footnote{
We shall give a simple proof for the case $H=K_r$ in Theorem \ref{thm:folklore}.}
and Reiher \cite{Re14} settled the case where $H$ is an odd cycle.
As a consequence, every cycle satisfies Conjecture~\ref{conj:KNRS}.

Provided that trees and cycles do satisfy Conjecture~\ref{conj:KNRS}, it is then natural to ask if we may add an edge to those examples. 
Our first main theorem states that this is indeed true.
We say that a graph is \emph{unicyclic} if it contains exactly one cycle as a subgraph. A \emph{chord} of a cycle is a pair of vertices that is not an edge of the cycle.
\begin{theorem}\label{thm:chord}
Conjecture~\ref{conj:KNRS} is true if $H$ is a unicyclic graph or a cycle plus a chord.
\end{theorem}

\medskip

The second example class is obtained by gluing complete multipartite graphs or odd cycles in a tree-like way, which we shall define precisely in Definition~\ref{def:Jdec}. To illustrate roughly at the moment, a \emph{$J$-decomposable} graph $H$ is a graph obtained by gluing vertex-disjoint copies of $J$ in a tree-like way. For example, the Goldner--Harary graph in Figure~\ref{fig:GH} is obtained by gluing six copies of $K_4$ in the illustrated tree-like way. 

In fact, some $J$-decomposable graphs have already been studied in different contexts. In particular, the $K_r$-decomposable graphs have played vital role as maximal graphs with a given tree-width, a graph parameter studied extensively since Robertson and Seymour's work~\cite{RS84}.

\begin{figure}
	\begin{center}
	\begin{tikzpicture}
\begin{scope}[every node/.style={circle,thick,draw}]
    \node (A) at (0,0) {A};
    \node (B) at (2,0) {B};
    \node (C) at (6,2.2) {C};
    \node (D) at (6,-2.2) {D};
    \node (E) at (6,0) {E};
    \node (F) at (5,1) {F};
    \node (G) at (5,-1) {G};
    \node (H) at (10,0) {H};
    \node (I) at (7,1) {I};
    \node (J) at (7,-1) {J};
    \node (K) at (12,0) {K};
\end{scope}

\begin{scope}[>={Stealth[black]},
              every node/.style={fill=white,circle},
              every edge/.style={draw=black, thick}]
    \path [-] (A) edge (B);
    \path [-] (A) edge (C);
    \path [-] (A) edge (D);
    \path [-] (B) edge (C);
    \path [-] (B) edge (D);
    \path [-] (B) edge (E);
    \path [-] (B) edge (F);
    \path [-] (B) edge (G);
    \path [-] (C) edge[bend right=270] (D);
    \path [-] (C) edge (E);
    \path [-] (C) edge (F);
    \path [-] (C) edge (I);
    \path [-] (C) edge (H);
    \path [-] (C) edge (K);
    \path [-] (D) edge (E);
    \path [-] (D) edge (G);
    \path [-] (D) edge (J);
    \path [-] (D) edge (H);
    \path [-] (D) edge (K);
    \path [-] (E) edge (F);
    \path [-] (E) edge (G);
    \path [-] (E) edge (H);
    \path [-] (E) edge (I);
    \path [-] (E) edge (J);
    \path [-] (H) edge (I);
    \path [-] (H) edge (J);
    \path [-] (H) edge (K);
\end{scope}
\end{tikzpicture}
\begin{tikzpicture}
\begin{scope}[every node/.style={circle,thick,draw}]
	\node (abcd) at (-1,0) {ABCD};
	\node (bcef) at (0,2) {BCEF};
	\node (bdeg) at (0,-2) {BDEG};
	\node (bcde) at (2,0) {BCDE};
	\node (cdeh) at (5,0) {CDEH};
	\node (cdhk) at (8,0) {CDHK};
	\node (cehi) at (7,2) {CEHI};
	\node (dehj) at (7,-2) {DEHJ};
\end{scope}
\begin{scope}[>={Stealth[black]},
              every node/.style={fill=white,circle},
              every edge/.style={draw=black, thick}]
    \path [-] (abcd) edge (bcde);
    \path [-] (bcef) edge (bcde);
    \path [-] (bdeg) edge (bcde);
    \path [-] (cdeh) edge (bcde);
    \path [-] (cdhk) edge (cdeh);
    \path [-] (cehi) edge (cdeh);
    \path [-] (dehj) edge (cdeh);
\end{scope}         
        
\end{tikzpicture}
\end{center}
\caption{The Goldner--Harary graph and its $K_4$-decomposition.}\label{fig:GH}
\end{figure}

\medskip

One may expect that the standard application of Cauchy--Schwarz inequality will prove that the graph $H$ obtained by a pairwise (symmetric) gluing of two copies of $J$ satisfies Conjecture~\ref{conj:KNRS}, once $J$ does.
Our result roughly states that it is possible to extend this standard Cauchy--Schwarz argument in a tree-like way, using the information theoretic approach developed in~\cite{CKLL15,CL16}. More details will be discussed in Section~\ref{sec:prelim} and we shall prove the following theorem in Section~\ref{sec:multipart} and~\ref{sec:oddcycles}.
\begin{theorem}\label{thm:treelike}
Conjecture \ref{conj:KNRS} is true if $H$ is $C_{2k+1}$-decomposable or $K(r_1,\cdots,r_\ell)$-decomposable.
\end{theorem}
We remark that Conlon and the author~\cite{CL16} proved that every $J$-decomposable graph has Sidorenko's property whenever $J$ is weakly norming. Since every even cycle is weakly norming, $C_{2k}$-decomposable graphs are already known to satisfy Conjecture~\ref{conj:KNRS}. Thus, Theorem~\ref{thm:treelike} together with the result in~\cite{CL16} implies that every $C_k$-decomposable graph satisfies Conjecture~\ref{conj:KNRS}.
Moreover, it is easy to check that Theorem~\ref{thm:chord} and~\ref{thm:treelike} are enough to complete the proof of Conjecture~\ref{conj:KNRS} for graphs with at most five vertices. 

\medskip

The proofs of Theorem~\ref{thm:chord} and~\ref{thm:treelike} use recently developed techniques to attack Sidorenko's conjecture in non-trivial ways. In particular, the information theoretic approach~\cite{CKLL15,CL16,KLL14,LSz12,Sz15} and the application of H\"older's inequality appeared in the very recent work~\cite{CL18} are the key ingredients.

\section{Preliminaries}\label{sec:prelim}

In what follows, let $|H|:=|V(H)|$, $e(H):=|E(H)|$, and $e_H[U]:=|E(H[U])|$ for $U\subseteq V(H)$. If $H$ is clear from the context, then we shall also write $e[U]=e_H[U]$.
Logarithms will always be understood to be base~2.
We denote by $P_\ell$ the $\ell$-edge path, to emphasise the number of edges and its parity.

The proof of Theorem \ref{thm:treelike} relies on the entropy analysis
applied in \cite{CKLL15,CL16}
and the proof of Theorem~\ref{thm:chord} also uses basic entropy inequalities, despite stated in terms of Jensen's inequality for logarithmic functions.
We state the following facts about entropy without proofs and refer the reader to \cite{AS08} for
more detailed information on entropy and conditional entropy.
 
\begin{lemma}\label{lem:entropy}
Let $X$, $Y$, and $Z$ be random variables and suppose that $X$ takes values in 
a set $S$, $\HH(X)$ is the entropy of $X$, and $\HH(X|Y)$ is the conditional entropy of $X$ given $Y$. Then 
\begin{enumerate}
	\item $\HH(X)\leq\log|S|$,
	\item $\HH(X|Y,Z)=\HH(X|Z)$ if $X$ and $Y$ are conditionally independent given $Z$.
\end{enumerate}
\end{lemma}

To formalise the definition $J$-decomposable graphs, 
it is convenient to use the notion of tree decompositions, introduced by Halin~\cite{Hal76}
and developed by Robertson and Seymour \cite{RS84}.
\begin{definition}
A \emph{tree decomposition} of a graph $H$ is a pair $(\mathcal{F}, \TT)$ consisting of a family $\mathcal{F}$  of vertex subsets of $H$ and  a tree $\TT$ on $\mathcal{F}$ such that
\begin{enumerate}
\item $\bigcup_{X\in\mathcal{F}}X=V(H)$, 
\item for each $e \in E(H)$, there exists a set $X \in \mathcal{F}$ such that
$X$ contains $e$, and
\item for $X,Y,Z\in \mathcal{F}$, $X\cap Y\subseteq Z$ 
whenever $Z$ lies on the path from $X$ to $Y$ in $\TT$.
\end{enumerate}
\end{definition}

\begin{definition}\label{def:Jdec}
Given a graph $H$ and an induced subgraph $J$,
a \emph{$J$-decomposition} of a graph $H$ is a tree decomposition $(\mathcal{F},\TT)$ of $H$ 
satisfying the following two extra conditions:
\begin{enumerate}
\item 
each induced subgraphs $H[X]$, $X \in \FF$, is isomorphic to $J$, and
\item 
for every pair $X,Y\in \FF$ which are adjacent in $\TT$, there is an isomorphism between the two copies $H[X]$ and $H[Y]$
of $J$ that fixes $X \cap Y$.
\end{enumerate}
We call a graph \emph{$J$-decomposable} if it allows a $J$-decomposition,
i.e., it can be obtained by symmetrically gluing copies of $J$ in a tree-like way.
If $J$ is a complete graph then the second condition on the symmetry between $H[X]$ and $H[Y]$, $XY\in E(\TT)$,
is automatically satisfied.
\end{definition}

We shall use the following simple lemma to count the number of edges in a $J$-decomposable graph. The proof will be given in the appendix.
\begin{lemma}\label{lem:edge_counts}
Let $(\FF,\TT)$ be a tree decomposition of a graph $H$. Then
\begin{align}\label{eq:edge_counts}
e(H)=\sum_{X\in\FF}e_H[X]-\sum_{XY\in E(\TT)}e_H[X\cap Y].
\end{align}
\end{lemma}

A folklore lemma, 
essentially implied by the Kolmogorov extension theorem, 
will be necessary to obtain an information-theoretic lemma that counts the number of $J$-decomposable graphs.
We give a simple proof when all the random variables take finitely many values,
which suffices for our purpose, in the appendix.
For a modern introduction to 
product measure spaces and the Kolmogorov extension theorem,
we refer to \cite{tao11}.

\begin{lemma}\label{lem:kolmogorov}
Let $(X_1,X_2)$ and $(X_2',X_3)$ be random vectors taking values in a finite set.
If $X_2$ and $X_2'$ are identically distributed, then there exists $(Y_1,Y_2,Y_3)$ 
such that
$Y_1$ and $Y_3$ are conditionally independent given $Y_2$,
$(X_1,X_2)$ and $(Y_1,Y_2)$ are identically distributed,
and $(X_2',X_3)$ and $(Y_2,Y_3)$ are also identically distributed.
\end{lemma}

Let $\FF$ be a family of subsets of $[k]:=\{1,2,\cdots,k\}$.
Partly motivated by the notion of tree decompositions,
a \emph{Markov tree} on $[k]$ is a pair $(\FF,\TT)$ with $\TT$ a tree on vertex set $\FF$ that satisfies
\begin{enumerate}
	\item $\bigcup_{F\in\mathcal{F}}F=[k]$ and
	\item for $A,B,C\in \mathcal{F}$, $A\cap B\subseteq C$ 
whenever $C$ lies on the path from $A$ to $B$ in $\TT$.
\end{enumerate}
Let $V$ be a finite set and for each $F\in\FF$ let $\X_F=(X_{i;F})_{i\in F}$ be a random vector 
	taking values in $V^{F}$.
We are interested in such random vectors where `local' information is `globally' extendible.
That is, there exist random variables $Y_1,Y_2,\cdots,Y_k$ such that, for each $F\in \FF$, the two random vectors $(Y_i)_{i\in F}$ and $\X_{F}$ are identically distributed over $V^{F}$, 
and thus, $(Y_i)_{i\in F}$ `copies' the given local information $\X_F$.
If such $Y_1,\cdots,Y_k$ exist, 
then $(X_{i;A})_{i\in A\cap B}$ and $(X_{j;B})_{j\in A\cap B}$ must be identically distributed.
The following theorem states that the converse
is also true and, moreover,
the maximum entropy under such constraints
can always be attained. We again postpone the proof until the appendix.

\begin{theorem}\label{thm:tree_entropy}
	Let $(\FF,\TT)$ be a Markov tree on $[k]$.
	Let $V$ be a finite set and for each $F\in\FF$ let $\X_F=(X_{i;F})_{i\in F}$ be a random vector 
	taking values in $V^F$.  
	If $(X_{i;A})_{i\in A\cap B}$ and $(X_{j;B})_{j\in A\cap B}$ are identically distributed
	whenever $AB\in E(\TT)$,
	then there exists $\Y=(Y_1,\cdots,Y_k)$
	with entropy
	\begin{align}\label{eq:tree_entropy}
		\HH(\Y)
		=\sum_{F\in\FF}\HH(\X_F)
		-\sum_{AB\in E(\TT)}\HH((X_{i;A})_{i\in A\cap B})
	\end{align}
	such that $(Y_i)_{i\in F}$ and $\X_F$ are identically distributed over $V^{F}$ for all $F\in\FF$.
\end{theorem}  

Let us discuss how Theorem \ref{thm:tree_entropy}
relates to the classical Cauchy--Schwarz inequality.
For a simple example, let $V=V(G)$ be the vertex set of a graph $G$, 
and let $(X,Y)$ and $(Y',Z)$ be two uniform random labelled edges.
Since the distributions of $Y$ and $Y'$ are identical, Theorem \ref{thm:tree_entropy} implies
that there exists $(X_1,X_2,X_3)$ with entropy
\begin{align*}
\HH(X_1,X_2,X_3)=\HH(X,Y)+\HH(Y',Z)-\HH(Y).
\end{align*}
As $(X_1,X_2)$ and $(X_2,X_3)$ are identically distributed with $(X,Y)$ and $(Y',Z)$, respectively, they are uniform random labelled edges in $G$.
Thus, $(X_1,X_2,X_3)$ is a homomorphic copy of $K_{1,2}$, where $X_2$ is the image of degree-two vertex.
Using basic facts on entropy
(see Lemma \ref{lem:entropy}),
we have
$\HH(X,Y)=\HH(Y',Z)=\log 2e(G) $, $\HH(Y)\leq\log|G|$, and $\HH(X_1,X_2,X_3)\leq\log|\Hom(K_{1,2},G)|$,
which implies $t_{K_{1,2}}(G)\geq t_{K_2}(G)^2$.
This is also an easy consequence of the Cauchy--Schwarz inequality
and furthermore, we may recover many
graph homomorphism inequalities obtained by using 
the Cauchy--Schwarz inequality
by the same argument above with different choice of $(X,Y,Y',Z)$.

In particular, this is same as using $|\FF|=2$ and $\TT$ the single edge tree for Theorem~\ref{thm:tree_entropy}.
Hence, \eqref{eq:tree_entropy} may be seen as a
tree-like extension of the Cauchy--Schwarz inequality.
In fact, analogous lemmas to Theorem \ref{thm:tree_entropy} have already been
used in \cite{CKLL15,CL16,KR11,Sz15} to obtain such results.

To obtain graph homomorphism inequalities,
the following corollary of Theorem \ref{thm:tree_entropy},
which appeared implicitly in \cite{CL16},
is useful. 
\begin{theorem}\label{thm:tree_hom}
	Let $G,H$, and $J$ be graphs.
	Suppose that $H$ is $J$-decomposable
	and $\Hom(J,G)$ is non-empty.
	Fix a $J$-decomposition $(\FF,\TT)$ of $H$.
	Then the following inequality holds:
	\begin{align}\label{eq:tree_hom}
		t_H(G)\geq \frac{t_{J}(G)^{|\FF|}}{\prod_{XY\in E(\TT)} t_{H[X\cap Y]}(G)}.
	\end{align}
\end{theorem}
\begin{proof}
	By definition, a $J$-decomposition $(\FF,\TT)$ is a Markov tree on $V(H)$.
	Let $\X_F=(X_{i,F})_{i\in F}$ be the uniform random homomorphism in $\Hom(J,G)$.
	Then both $(X_{i,A})_{i\in A\cap B}$ and
	$(X_{j,B})_{j\in A\cap B}$ are 
	supported on the set $\Hom(H[A\cap B],G)$ and moreover, they are identically distributed.
	This is because the distributions are projected from the uniform distribution on $\Hom(J,G)$
	in the same way, as there exists an isomorphism between $H[A]$ and $H[B]$ that fixes $A\cap B$.
	Thus, by Theorem \ref{thm:tree_entropy},
	there exists $(Y_v)_{v\in V(H)}$ such that
	\begin{align}\label{eq:apply}
		\HH((Y_v)_{v\in V(H)})
		=|\FF|\log|\Hom(J,G)|
		-\sum_{AB\in E(\TT)}\HH((X_{i;A})_{i\in A\cap B}),
	\end{align}
	where each $(Y_u)_{u\in F}$ is identically distributed with $\X_F$.
	Since each $\X_{F}$ takes values in $\Hom(J,G)$, $(Y_u,Y_v)$ is always an ordered edge in $G$.
	Thus, $(Y_v)_{v\in V(H)}$ is a (not necessarily uniform) random homomorphism from $H$ to $G$.
	Now \eqref{eq:apply} gives
	\begin{align*}
	\log|\Hom(H,G)|
	&\geq
	|\FF|\log|\Hom(J,G)|-\sum_{AB\in E(\TT)}\HH((X_{i;A})_{i\in A\cap B})\\
	&\geq|\FF|\log|\Hom(J,G)|-\sum_{AB\in E(\TT)}\log|\Hom(H[A\cap B],G)|.
	\end{align*}
	By rescaling by subtracting $|H|\log|G|$ on both sides and 
	using the identity \eqref{eq:edge_counts},
	we obtain the inequality~\eqref{eq:tree_hom}.
\end{proof}

Another technical ingredient we need is Reiher's lemma~\cite{Re14}, which enables the continuous relaxation of the $(\rho,d)$-denseness condition.
\begin{lemma}[Lemma 2.1 in \cite{Re14}]\label{lem:Reiher}
Let $G$ be a $(\rho,d)$-dense graph on $n$ vertices, and let $f:V(G)\rightarrow [0,1]$ be a function satisfying
$\sum_{v\in V(G)}f(v)\geq \rho n$.
Then
\begin{align}\label{eq:relax}
\sum_{(u,v)\in \Hom(K_2,G)}f(u)f(v)\geq d\left(\sum_{v\in V(G)}f(v)\right)^2-2n.
\end{align}
\end{lemma}
This lemma is often used in the following form.
\begin{corollary}\label{cor:Reiher}
Let $G$ be a $(\rho,d)$-dense graph on $n$ vertices and let $f(v;\xx):V(G)\rightarrow [0,1]$ be a function associated with a $k$-tuple $\xx\in V(G)^k$. 
If $\sum_{(v,\xx)\in V(G)^{k+1}}f(v;\xx)\geq \alpha n^{k+1}$,
then
\begin{align*}
   \sum_{\xx\in V(G)^k}\sum_{(u,v)\in \Hom(K_2,G)}f(u;\xx)f(v;\xx)
   \geq d\left(\alpha^2 -2\alpha \rho-2/n\right)n^{k+2}.
\end{align*}
\end{corollary}
\begin{proof}
Let $S_\rho$ be the set of $k$-tuples $\xx$ in $V(G)^k$ such that $\sum_{v\in V(G)^{k+1}}f(v;\xx)\geq \rho n$. Then by Lemma~\ref{lem:Reiher}, for each $\xx\in S_\rho$,
\begin{align*}
\sum_{(u,v)\in \Hom(K_2,G)}f(u;\xx)f(v;\xx)\geq d\left(\sum_{v\in V(G)}f(v;\xx)\right)^2-2n.
\end{align*}
Thus, summing over all $\xx\in S_\rho$ gives
\begin{align*}
\sum_{\xx\in S_\rho}\sum_{(u,v)\in \Hom(K_2,G)}f(u;\xx)f(v;\xx)
&\geq d\sum_{\xx\in S_\rho}\left(\sum_{v\in V(G)}f(v;\xx)\right)^2-2n|S_\rho|\\
&\geq \frac{d}{|S_\rho|}\left(\sum_{\xx\in S_\rho}\sum_{v\in V(G)}f(v;\xx)\right)^2-2n|S_\rho|,
\end{align*}
where the last inequality follows from convexity. As $\sum_{\xx\notin S_\rho}\sum_{v\in V(G)}f(v;\xx)\leq \rho n^{k+1}$ and $|S_\rho|\leq n^{k}$,
\begin{align*}
\sum_{\xx\in S_\rho}\sum_{(u,v)\in \Hom(K_2,G)}f(u;\xx)f(v;\xx)
&\geq \frac{d}{|S_\rho|}\left(\alpha-\rho\right)^2n^{2k+2}-2n|S_\rho|\\
&\geq d\left(\alpha^2 -2\alpha \rho-2/n\right)n^{k+2}.
\end{align*}
This is also a lower bound for $\sum_{\xx\in V(G)^k}\sum_{(u,v)\in \Hom(K_2,G)}f(u;\xx)f(v;\xx)$ and hence completes the proof.
\end{proof}

\section{Adding an edge to cycles and trees}
The two main ingredients in proving Theorem~\ref{thm:chord} are Li--Szegedy smoothness introduced in~\cite{LSz12} and H\"older's inequality.
To sketch very roughly, the crucial caveat in using the local denseness or its relaxation, Lemma~\ref{lem:Reiher}, is the symmetry. That is, in \eqref{eq:relax}, $f(u)$ and $f(v)$ must be the \emph{same} function $f$ evaluated at distinct vertices. If the asymmetry in~\eqref{eq:relax} is allowed, then it closely resembles more general condition so-called bi-denseness, and an analogue of Conjecture~\ref{conj:KNRS} will be an easy consequence.

The starting point is to observe that a recent idea, the `H\"older trick' used in~\cite{CL18}, provides some symmetrisation. However, the H\"older trick inevitably produces rational exponents on the functions, i.e., the normalised number of graphs supported on fixed vertices, in the expectation.
These rational exponents are in general hard to control, especially if they are smaller than one and hence allow no convexity inequalities. However, Li--Szegedy smoothness, which will be introduced shortly, enables us to handle them to deduce Theorem~\ref{thm:chord}.

\medskip

Let $U$ be a vertex subset of $H$. For each $\phi\in \Hom(H[U],G)$, denote by $t_H(G;\phi)$ be the normalised number of homomorphisms in $\Hom(H,G)$ that extend $\phi$, i.e.,
\[
t_H(G;\phi):=\frac{\left|\{\psi\in\Hom(H,G):\psi|_{U}=\phi\}\right|}{|G|^{|H|-|U|}}.
\]
If $U$ consists of a single vertex or a pair of vertices that are mapped to $x\in V(G)$ or $(x,y)\in V(G)^2$ and is already specified in the context, we also write $t_H(G;x)$ or $t_H(G;x,y)$, respectively.

We say that $U$ is \emph{smooth} in $H$ if there exists a probability distribution $p:\Hom(H[U],G)\rightarrow [0,1]$ such that
\begin{align}\label{eq:smooth}
   & \sum_{\phi\in\Hom(H[U],G)} p_\phi\log t_H(G;\phi)\geq (e(H)-e[U])\log t_{K_2}(G)
    \\~\text{ and }~
   &\sum_{\phi\in\Hom(H[U],G)} p_\phi\log (1/p_\phi)-|U|\log|G|\geq e[U]\log t_{K_2}(G),\label{eq:large_entropy}
\end{align}
where $p_\phi=p(\phi)$. For brevity, we also say that the induced subgraph $H[U]$ is \emph{smooth} in $H$ whenever $U$ is.
We also say that such $p$ \emph{realises} the smoothness of $U$ (or $H[U]$) in $H$.
The following theorem essentially appeared in~\cite{LSz12}, although we state it in a slightly more general way. We give a short proof for completeness.
\begin{lemma}[\cite{LSz12}, Corollary 3.1]
If there exists $U\subseteq V(H)$ that is smooth in $H$, then $H$ has Sidorenko's property.
\end{lemma}
\begin{proof}
As $\{\psi\in\Hom(H,G):\psi|_{U}=\phi\}$, $\phi\in \Hom(H[U],G)$, partitions $\Hom(H,G)$,
\begin{align}\label{eq:partition}
   |G|^{|U|} t_H(G)=\sum_{\phi\in\Hom(H[U],G)} t_H(G;\phi).
\end{align}
We denote by $\EE_p[\cdot]$ the expectation taken by the probability distribution $p$ that realises the smoothness of $U$ in $H$.
Then
\begin{align*}
    \log\left(\sum_{\phi\in\Hom(H[U],G)} t_H(G;\phi)\right)&=
    \log\EE_p \left[\frac{t_H(G;\phi)}{p_\phi}\right]\\
    &\geq\EE_p\Big[\log t_H(G;\phi) +\log(1/p_\phi)\Big]\\
    &\geq e(H)\log t_{K_2}(G)+ |U|\log|G|,
\end{align*}
where the first inequality follows from concavity of logarithm and the second from the smoothness of $U$ in $H$. Comparing this bound with \eqref{eq:partition} concludes the proof.
\end{proof}

As a partial converse, Li and Szegedy proved that Sidorenko's property implies the smoothness of an edge.
\begin{theorem}[\cite{LSz12}, Theorem~5]\label{thm:edge_smooth}
    If $H$ has Sidorenko's property, then every edge in $E(H)$ is smooth in $H$.
\end{theorem}
If $p$ realises the smoothness of an edge $e$ in $H$, then \eqref{eq:large_entropy} reduces to
\begin{align*}
    \sum_{\phi\in\Hom(K_2,G)} p_\phi\log(1/p_{\phi})\geq \log |\Hom(K_2,G)|=\log 2e(G) .
\end{align*}
The left hand side is the entropy of the distribution $p$, whose maximum value is the right hand side. Thus, the equality holds and moreover, $p$ must be the uniform distribution on $\Hom(K_2,G)$.
In other words, the smoothness of an edge in $H$ is equivalent to
\begin{align}\label{eq:smooth_edge}
    \EE \big[\log t_H(G;\phi)\big]\geq (e(H)-1)\log t_{K_2}(G),
\end{align}
where the expectation is taken over the set of uniform random ordered edges $\phi(e)$ in $G$.

We are now ready to prove the first half of Theorem~\ref{thm:chord}.
\begin{theorem}
Conjecture~\ref{conj:KNRS} is true if $H$ is a cycle plus a chord.
\end{theorem}
\begin{proof}
Let $G$ be an $n$-vertex $(\rho,d)$-dense graph, where $\rho$ will be specified later in the proof.
We may assume that $H$ is the graph obtained by adding a chord edge to a cycle $C_k$, where $V(C_k)=\mathbb{Z}_k$ and $E(C_k)=\{\{i,i+1\}:i=1,2,\cdots,k\}$.
All the additions to represent vertex labels in $C_k$ will be taken modulo $k$ throughout the proof, e.g., $k-r=-r$.

Suppose firstly that $k=2m$ is an even integer.
If the chord edge produces two shorter even cycles, then by Theorem~2 in~\cite{LSz12}, $H$ has Sidorenko's property.
Otherwise, $H$ has two odd cycles $C_{2r+1}$ and $C_{2m-2r+1}$ on $\{-r,-r+1,\cdots,r\}$ and $\{r,r+1,\cdots,2m-r\}$, respectively, i.e., the chord edge is $\{-r,r\}$.

We shall embed the two antipodal vertices $0$ and $m$ to two fixed vertices $x,y$ in $G$ and count the normalised number of homomorphic $H$-copies by using Lemma~\ref{lem:Reiher}, denoted by $t_H(G;x,y)$.
Let $f(v;x,y)$ be the normalised number of the $m$-edge (directed) walks from $x$ to $y$ such that $v$ is the $(r+1)$-th vertex in each walk. Here the normalisation means dividing the number of such paths by $n^{m-2}$ to guarantee $0\leq f(v;x,y)\leq 1$. 
Then, Sidorenko's property of paths\footnote{This was obtained multiple times independently by Mullholland--Smith~\cite{MS59}, Atkinson--Watterson--Moran~\cite{AWM60}, and Blakley--Roy~\cite{BR65}, although often cited as the Blakley--Roy inequality.} gives
\begin{align}\label{eq:pathcount}
    \sum_{x,y,v\in V(G)} f(v;x,y)= n^3 t_{P_{m}}(G)\geq d^{m}n^3.
\end{align}
Hence, Corollary~\ref{cor:Reiher} implies
\begin{align*}
    \sum_{(x,y)\in V(G)^2} t_H(G;x,y)&=\sum_{x,y\in V(G)}\sum_{(u,v)\in \Hom(K_2,G)}f(u;x,y) f(v;x,y)\\
    &\geq d(d^m-\rho)^2 n^2 - 2n\\
    &\geq d^{2m+1}\left(1-2\rho/d^{m}-2/nd^m\right)n^2. 
\end{align*}
Thus, taking $\rho=\eta d^{m}/4$ and $n>4/\eta d^m$ gives the desired bound $t_H(G)\geq d^{2m+1}(1-\eta)$.
\medskip

Suppose now that $k=2m+1$ is an odd integer. Then adding a chord to $C_k$ always makes two shorter cycles $C_{2r+1}$ and $C_{2m-2r+2}$ with distinct parity. We may assume that the chord edge is $\{-r,r\}$, $V(C_{2r+1})=\{-r,-r+1,\cdots,r\}$ and $V(C_{2m-2r})=\{r,r+1,\cdots,2m+1-r\}$.
Denote by $h(x,y)$ the number of $(2m-2r+1)$-edge walks from $x$ to $y$ divided by $n^{2m-2r}$ if $(x,y)$ is an (ordered) edge in $G$. Otherwise let $h=0$. Let $g(x,y)$ be the edge indicator function of $G$. Then
\begin{align*}
    t_H(G)=\EE\left[h(x_r,x_{-r})g(x_r,x_{-r})\prod_{i=-r}^{r-1}g(x_i,x_{i+1})\right],
\end{align*}
where the expectation is taken over the uniform random vertices $x_{i}, -r\leq i\leq r$. By the natural symmetry in $C_{2r+1}$, we may also write
\begin{align*}
    t_H(G)=\EE\left[h(x_j,x_{j+1})g(x_r,x_{-r})\prod_{i=-r}^{r-1}g(x_i,x_{i+1})\right]
\end{align*}
for any $-r\leq j\leq r-1$. Denote by $F_j$ the product of functions in the expectation above, i.e., $F_j:=h(x_j,x_{j+1})g(x_r,x_{-r})\prod_{i=-r}^{r-1}g(x_i,x_{i+1})$, so that $t_H(G)=\EE[F_j]$.
Then H\"{o}lder's inequality gives
\begin{align*}
    t_H(G) =\left( \prod_{j=-r}^{r-1} \EE [F_j ]\right)^{\frac{1}{2r}}
    \geq \EE\left[ \prod_{j=-r}^{r-1}  F_j^{\frac{1}{2r}}\right]
    = \EE \left[g(x_{-r},x_r)\prod_{i=-r}^{r-1}g(x_i,x_{i+1})h(x_i,x_{i+1})^{\frac{1}{2r}}\right].
\end{align*}
We claim that 
\begin{align}
    \EE \left[\prod_{i=0}^{r-1}g(x_i,x_{i+1})h(x_i,x_{i+1})^{\frac{1}{2r}}\right] \geq \left(m+\frac{1}{2}\right)\log t_{K_2}(G).
\end{align}
Suppose that the claim is true.
We may rewrite the lower bound from H\"older's inequality as
\begin{align}\label{eq:rewrite}
    t_H(G)\geq \EE \left[g(x_{-r},x_r)\prod_{i=0}^{r-1}g(x_i,x_{i+1})h(x_i,x_{i+1})^{\frac{1}{2r}}\prod_{i=0}^{-r+1}g(x_i,x_{i-1})h(x_i,x_{i+1})^{\frac{1}{2r}}\right].
\end{align}
This allows us to apply Corollary~\ref{cor:Reiher}
with the choice $$f(v;\xx)=\EE\left[\prod_{i=0}^{r-1}g(x_i,x_{i+1})h(x_i,x_{i+1})^{\frac{1}{2r}}\middle\vert x_0=\xx, x_r=v \right]$$ 
and $\alpha= d^{m+1/2}$, since the lower bound in~\eqref{eq:rewrite} is exactly $\EE\left[f(v;\xx)f(u;\xx)g(u,v)\right]$.
Thus, by Corollary~\ref{cor:Reiher},
\begin{align*}
    t_H(G) \geq d^{2m+1}(1-2\rho/d^{m+1/2}-2/nd^{2m}).
\end{align*}
Taking $\rho=1/4\eta d^{m+1/2}$ and $n>4/\eta d^{2m}$ completes the proof.
\medskip

It remains to prove the claim. By Sidorenko's property of paths, we have
\begin{align}\label{eq:path_exponent}
    \EE \left[\prod_{i=0}^{r-1}g(x_i,x_{i+1})h(x_i,x_{i+1})^{\frac{1}{2r}}\right]\geq 
     \EE \left[g(x,y)h(x,y)^{\frac{1}{2r}}\right]^{r}.
\end{align}
Let $\tilde{\EE}[.]$ be the expectation taken by the distribution of a uniform random (labelled) edge, i.e., $\tilde{\EE}[f]=\frac{1}{t_{K_2}(G)}\EE[g(x,y)f]$ whenever $f$ is supported on the set of edges. We may then write
\begin{align*}
    \EE \left[g(x,y)h(x,y)^{\frac{1}{2r}}\right]
    =t_{K_2}(G)\tilde{\EE} \left[h(x,y)^{\frac{1}{2r}}\right]
    =t_{K_2}(G)\tilde{\EE} \left[t_{C_{2m-2r+2}}(G;x,y)^{\frac{1}{2r}}\right].
\end{align*}
Hence, by using concavity of logarithm, we obtain 
\begin{align*}
    \log \EE \left[g(x,y)h(x,y)^{\frac{1}{2r}}\right]
    &=\log t_{K_2}(G) +\log \tilde{\EE} \left[t_{C_{2m-2r+2}}(G;x,y)^{\frac{1}{2r}}\right]\\
    &\geq \log t_{K_2}(G) +\frac{1}{2r}\tilde{\EE} \left[\log t_{C_{2m-2r+2}}(G;x,y)\right]
\end{align*}
Since even cycles have Sidorenko's property,\footnote{Sidorenko~\cite{Sid93} firstly proved this fact, though it was already implicitly proven in Chung--Graham--Wilson's definition of quasirandomness~\cite{CGW89}. It is also reproved in ~\cite{H10, LSz12} since then.} 
a pair of vertices that induce an edge is smooth in the even cycle $C_{2m-2r+2}$ by Theorem~\ref{thm:edge_smooth}.
The smoothness of an edge, as rephrased in~\eqref{eq:smooth_edge}, then yields
\begin{align*}
    \tilde{\EE} \left[\log t_{C_{2m-2r+2}}(G;x,y)\right] \geq (2m-2r+1)\log t_{K_2}(G).
\end{align*}
Together with~\eqref{eq:path_exponent}, this concludes the proof of the claim.
\end{proof}

To prove Conjecture~\ref{conj:KNRS} for unicyclic graphs, we need the following fact, already appeared in~\cite{LSz12} and explained in terms of information theory in~\cite{CKLL15,Sz15}.
\begin{theorem}[(7) in~\cite{LSz12}]\label{thm:subtree_smoothness}
Let $T$ be a tree. Then every induced subtree $S$ is smooth in $T$.
Moreover, the smoothness of $S$ in $T$ is realised by the same distribution that realises the smoothness of $S$ in~$T'$ for any tree $T'$ that contains $S$ as an induced subtree.  
\end{theorem}
More precisely, the smoothness of an induced subtree in $T$ is always realised by the distribution induced by the \emph{tree branching random walk}, that is, starting from a uniform random edge and branching uniformly at random to obtain a homomorphic copy of a tree. For more details, see~\cite{CKLL15}.

\begin{theorem}\label{thm:unicycle}
Conjecture~\ref{conj:KNRS} is true if $H$ is unicyclic.
\end{theorem}
\begin{proof}
    If the unique cycle in $H$ is of even length, then $H$ is bipartite.
    Moreover, $H$ has Sidorenko's property, as shown in Theorem~2 in~\cite{Sid92} and thus the conclusion easily follows.
    
    Thus, we may assume that the unique cycle in $H$ is of odd length. Let $-m,-m+1,\cdots,m$ be vertices in $V(H)$ that induce an odd cycle of length $2m+1$ with edges $\{i,i+1\}$, $-m\leq i\leq m$, where the addition is taken modulo $2m+1$. Throughout the proof, all the additions of vertex labels will be taken modulo $2m+1$.
    
    Let $T_i$ be the subgraph of $H$ induced on $V(H)\setminus ([2m+1]\setminus\{i\})$. In particular, $T_i$ is a tree that contains the unique vertex $\{i\}$ from the cycle.
    Let $\tau_i(x):=t_{T_i}(G;x)$ for $x\in V(G)$, that is, the number of homomorphisms from $T_i$ to $G$ that maps $i$ to $x$ divided by $|G|^{|T_i|-1}$. Denote by $g$ the edge indicator of $G$.
    We may then write
    \begin{align*}
        t_H(G) = \EE\left[\prod_{i=-m}^{m}g(x_i,x_{i+1})\tau_i(x_i)\right],
    \end{align*}
    or, by the symmetry of $C_{2m+1}$ that maps $i$ to $-i$,
    \begin{align*}
        t_H(G) = \EE\left[\prod_{i=-m}^{m}g(x_i,x_{i+1})\tau_{-i}(x_i)\right].
    \end{align*}
   Applying the Cauchy--Schwarz inequality then yields
    \begin{align*}
        t_H(G)\geq \EE\left[ \prod_{i=-m}^{m} g(x_i,x_{i+1})\prod_{j=-m}^{m} \sqrt{\tau_{-j}(x_j)\tau_{j}(x_j)}\right].
    \end{align*}
    We aim to rephrase above to apply Corollary~\ref{cor:Reiher}. To this end, let $\sigma_j(x):=\sqrt{\tau_j(x)\tau_{-j}(x)}$ so that $\sigma_j=\sigma_{-j}$. Then the right-hand side above is
    \begin{align}\label{eq:symmetrised}
    \EE\left[g(x_{-m},x_{m})\tau_{0}(x_{0})
    \prod_{i=1}^{m}\sigma_{i}(x_{i})g(x_i,x_{i-1})
    \prod_{j=1}^{m}\sigma_{j}(x_{-j})g(x_{-j},x_{-j+1}) \right].
    \end{align}
    We claim that
    \begin{align}\label{eq:claim_unicyclic}
        \EE\left[\sqrt{\tau_{0}(x_{0})}    \prod_{i=1}^{m}\sigma_{i}(x_{i})g(x_i,x_{i-1})\right]\geq 
        \frac{1}{2}\left(e(H)-1\right)\log t_{K_2}(G)
    \end{align}
    If the claim is true, then by letting
    \begin{align*}
        f(v;\xx)=\EE\left[\sqrt{\tau_{0}(x_{0})}    \prod_{i=1}^{m}\sigma_{i}(x_{i})g(x_i,x_{i-1})\middle\vert x_0=\xx, x_m=v\right]
    \end{align*}
    and $\alpha=d^{\frac{1}{2}(e(H)-1)}$ in Corollary~\ref{cor:Reiher},
    \eqref{eq:symmetrised} gives
    \begin{align*}
        t_H(G)\geq \EE\big[f(u;\xx)f(v;\xx)g(u,v)\big]\geq d^{e(H)}(1-2\rho/d^{\frac{1}{2}(e(H)-1)}-2/nd^{e(H)-1}).
    \end{align*}
    Taking $\rho=\eta d^{\frac{1}{2}(e(H)-1)}/4$ and $n>4\eta/d^{e(H)-1}$ therefore finishes the proof.
    \medskip
    
    It remains to verify the claim. For $i\geq 0$, let $T_i'$ be the tree induced on $V(T_i)\cup\{0,1,\cdots,m\}$ and let $T_{-i}'$ be the tree obtained by identifying $-i\in V(T_{-i})$ and the vertex $i$ on the $m$-edge path $P$ on $\{0,1,\cdots,m\}$.
    By Theorem~\ref{thm:subtree_smoothness}, $V(P)=[m]$ is smooth in $T_i'$ 
    and moreover, its smoothness is realised by the same distribution $p(x_0,x_1,\cdots,x_m)$ on $\Hom(P_m,G)$ regardless of $-r\leq i\leq r$.
    Now write the left-hand side of~\eqref{eq:claim_unicyclic} as
    \begin{align*}
        \EE\left[\sqrt{\tau_{0}(x_{0})}    \prod_{i=1}^{m}\sigma_{i}(x_{i})g(x_i,x_{i-1})\right]
        =\tilde{\EE}\left[\frac{\sqrt{\tau_{0}(x_{0})}    \prod_{i=1}^{m}\sigma_{i}(x_{i})g(x_i,x_{i-1})}{p(x_0,x_1,\cdots,x_m)}\right],
    \end{align*}
    where $\tilde{\EE}[.]=\EE[p(.)]$. This is possible since $\sqrt{\tau_{0}(x_{0})}    \prod_{i=1}^{m}\sigma_{i}(x_{i})g(x_i,x_{i-1})=0$
    whenever $p=0$. Then by concavity of logarithm,
    \begin{align*}
        &\log \tilde{\EE}\left[\frac{\sqrt{\tau_{0}(x_{0})}    \prod_{i=1}^{m}\sigma_{i}(x_{i})g(x_i,x_{i-1})}{p(x_0,x_1,\cdots,x_m)}\right]
       \geq \tilde{\EE}\left[\log\left(\frac{\sqrt{\tau_{0}(x_{0})}    \prod_{i=1}^{m}\sigma_{i}(x_{i})g(x_i,x_{i-1})}{p(x_0,x_1,\cdots,x_m)}\right)\right]\\
       &=\frac{1}{2}\sum_{i=0}^{m}\tilde{\EE}\Big[\log \tau_i(x_i) \Big]
       +\frac{1}{2}\sum_{i=1}^{m}\tilde{\EE}\Big[\log\tau_{-i}(x_i) \Big]
       +\tilde{\EE}\left[\log \prod_{i=1}^m g(x_i,x_{i-1}) \right]
       +\tilde{\EE}\left[\log\left(\frac{1}{p(x_0,\cdots,x_m)}\right)\right].
    \end{align*}
    Firstly, as $p$ is supported on $\Hom(P_m,G)$, $\prod_{i=1}^m g(x_i,x_{i-1})=1$ whenever $p$ is nonzero. Thus, the term $\tilde{\EE}\big[\log \prod_{i=1}^m g(x_i,x_{i-1}) \big]$ is always zero. Secondly, $\tilde{\EE}\left[\log\left(\frac{1}{p(x_0,\cdots,x_m)}\right)\right]$ is the entropy of the distribution $p$, which is at least $m\log t_{K_2}(G)$ by~\eqref{eq:large_entropy}. 
    Observe that, for each $i\geq 0$,
    \begin{align*}
       & \tilde{\EE}\Big[\log \tau_i(x_i)\Big]=\EE\Big[p(x_0,\cdots,x_m)\log t_{T_i'}(G;x_0,\cdots,x_m)\Big]\text{ and }\\
        &\tilde{\EE}\Big[\log \tau_{-i}(x_i)\Big]=\EE\Big[p(x_0,\cdots,x_m)\log t_{T_{-i}'}(G;x_0,\cdots,x_m)\Big].
    \end{align*}
    Indeed, $\tau_i(x_i)$ counts the normalised number of homomorphic copies of $T_i$ such that $i$ is mapped to $x_i$,
    which is equal to the normalised number of homomorphic copies of $T_i'$ such that $P$ is mapped to the (homomorphic) $m$-edge path on $\{x_0,x_1,\cdots,x_m\}$.
    By the same reason, $\tau_{-i}(x_i)$ is the normalised number of homomorphic copies of $T_{-i}'$ supported on the $m$-edge path on $\{x_0,\cdots,x_m\}$.
    Therefore, by the smoothness condition~\eqref{eq:smooth} of $V(P)=[m]$ in $T_i'$ or in $T_{-i}'$, it follows that
    \begin{align*}
        \tilde{\EE}\Big[\log \tau_i(x_i) \Big]\geq e(T_i)\log t_{K_2}(G)~\text{ and }~
        \tilde{\EE}\Big[\log\tau_{-i}(x_i) \Big] \geq e(T_{-i})\log t_{K_2}(G)
    \end{align*}
    for each $0\leq i\leq m$. Therefore, we obtain
    \begin{align*}
        \EE\left[\sqrt{\tau_{0}(x_{0})}    \prod_{i=1}^{m}\sigma_{i}(x_{i})g(x_i,x_{i-1})\right]
        \geq \left(m+\frac{1}{2}\sum_{i=-m}^{m}e(T_i)\right)\log t_{K_2}(G),
    \end{align*}
    which, together with the fact $e(H)=2m+1+\sum_{i=-m}^{m} e(T_i)$, proves~\eqref{eq:claim_unicyclic}.
\end{proof}

\section{Counting $K(r_1,r_2,\cdots,r_\ell)$-decomposable graphs}\label{sec:multipart}

In this section, we prove the following half of Theorem~\ref{thm:treelike}.

\begin{theorem}\label{thm:multipartite}
Let $r_1,r_2,\cdots,r_\ell$ be non-negative integers.
Then Conjecture \ref{conj:KNRS} is true if $H$ is a $K(r_1,r_2,\cdots,r_\ell)$-decomposable graph.
\end{theorem}

As a warm-up, we begin by proving the folklore fact that 
Conjecture \ref{conj:KNRS} is true for the case $H=K_r$.
It will also be technically helpful in what follows.
\begin{theorem}\label{thm:folklore}
Given $\eta,d>0$ and positive integer $r$, there exists $\rho=\rho(\eta,d,r)>0$ such that 
		\begin{align*}
			t_{K_r}(G)\geq (1-\eta)d^{r(r-1)/2}
		\end{align*}			
	for every sufficiently large $(\rho,d)$-dense graph $G$.
\end{theorem}

This follows immediately from a recursive statement,
which states that we may add an 
apex vertex to any graph satisfying
Conjecture \ref{conj:KNRS} to obtain another:
\begin{theorem}\label{thm:apex}
Let $\widehat{H}$ be the graph obtained by adding a vertex to $H$ 
which is adjacent to all vertices in $H$.
If Conjecture \ref{conj:KNRS} is true for $H$,
then it is also true for $\widehat{H}$.
\end{theorem}
\begin{proof}
Let $\rho>0$ be such that
\begin{align*}
 t_{H}(G)\geq (1-\eta/2)d^{e(H) }
\end{align*}
whenever $J$ is a sufficiently large $(\rho,d)$-dense graph.
We may assume that $\rho\leq\eta d/(2|H|)$.
Let $G$ be a $(\rho^2,d)$-dense graph on $n$ vertices.
Denote by $U$ the set of vertices $v$ in $G$ such that
$\deg(v) \geq \rho n$. 
Let $a$ be the apex vertex in $\widehat{H}$
and let $c(v)$ be the number of homomorphisms $\phi$ from $\widehat{H}$ to $G$ such that $\phi(a)=v$.
Observe that for any $W\subseteq V(G)$ of size $|W|\geq \rho n$,
the induced subgraph $G[W]$ is $(\rho,d)$-dense.
Thus,
\begin{align}\label{eq:Hhat}
	|\Hom(\widehat{H},G)|
	&=\sum_{v\in V(G)} c(v)
	\geq \sum_{u\in U} |\Hom(H,G[N(u)])|\nonumber\\
	&\geq \sum_{u\in U}(1-\eta/2)d^{e(H) }|N(u)|^{|H|}\nonumber\\
	&\geq \frac{(1-\eta/2)d^{e(H) }}{|U|^{|H|-1}}\left(\sum_{u\in U}|N(u)|\right)^{|H|},
\end{align}
where the last inequality follows from convexity.
Note that
\begin{align}\label{eq:NU}
\sum_{u\in U}|N(u)|&= 2e(G) -\sum_{v\notin U}|N(v)|\\
&\geq (d-\rho)n^2\geq \left(1-\frac{\eta}{2|H|}\right)dn^2,
\end{align}
where the last inequality follows from $\rho\leq\eta d/(2|H|)$.
By combining \eqref{eq:Hhat} and \eqref{eq:NU}, we obtain
\begin{align*}
	|\Hom(\widehat{H},G)|
	&\geq (1-\eta/2)\left(1-\frac{\eta}{2|H|}\right)^{|H|}d^{|H|+e(H) }n^{|H|+1}\\
	&\geq (1-\eta)d^{e(\widehat{H})}n^{|\widehat{H}|}.\tag*{\qedhere} 
\end{align*} 
\end{proof}

However, the classical approach above does not give a tight enough comparison  between~$t_{K_{r+1}}(G)$ and~$t_{K_{r}}(G)$.
When applying Theorem \ref{thm:tree_hom},
the main difficulty often lies in
bounding the terms $t_{H[X\cap Y]}(G)$ from above
in terms of $t_J(G)$.
The following lemma gives the control needed to  
prove Theorem \ref{thm:multipartite}. 
\begin{lemma}\label{lem:multi}
Given $\delta>0$ and positive integers $\ell,r_1,r_2,\cdots, r_\ell$,
let $r=\sum_{i=1}^\ell r_i$.
Then there exists $\rho=\rho(\delta,d,r_1,r_2,\cdots,r_\ell)$ such that
\begin{align*}
	t_{K(r_1,r_2,\cdots,r_\ell)}(G)\geq
	(1-\delta)d^{r-r_1}
	t_{K(r_1-1,r_2,\cdots,r_\ell)}(G)
\end{align*}
for every sufficiently large $(\rho,d)$-dense graph $G$.
\end{lemma}
\begin{proof}
Suppose $r_1\geq 2$.
Then the complete $\ell$-partite graph $K(r_1,r_2,\cdots,r_\ell)$ can be obtained by
gluing two copies of $K(r_1-1,r_2,\cdots,r_\ell)$
along their subgraphs induced on each vertex set minus a vertex in the vertex class of size $r_1-1$ in the $\ell$-partition,
which is isomorphic to $K(r_1-2,r_2,\cdots,r_\ell)$.
Hence, by the Cauchy--Schwarz inequality or Theorem \ref{thm:tree_hom} with $|\FF|=2$,
we have
\begin{align*}
t_{K(r_1,r_2,\cdots,r_\ell)}(G)
\geq \frac{t_{K(r_1-1,r_2,\cdots,r_\ell)}(G)^2}{t_{K(r_1-2,r_2,\cdots,r_\ell)}(G)}.
\end{align*}
Here we do not worry about the case $t_{K(r_1-2,r_2,\cdots,r_\ell)}(G)=0$, because by Theorem \ref{thm:folklore}
it is always positive.
Repeating this inequality gives the following log-convexity:
\begin{align}\label{eq:logconvex}
\frac{t_{K(r_1,r_2,\cdots,r_\ell)}(G)}{t_{K(r_1-1,r_2,\cdots,r_\ell)}(G)}
\geq
\frac{t_{K(r_1-1,r_2,\cdots,r_\ell)}(G)}{t_{K(r_1-2,r_2,\cdots,r_\ell)}(G)}
\geq\cdots\geq 
\frac{t_{K(1,r_2,\cdots,r_\ell)}(G)}{t_{K(r_2,\cdots,r_\ell)}(G)}.
\end{align}
Thus, the goal reduces to the case $r_1=1$.
We claim that, given $\eta>0$, there exists $\rho=\rho(\eta,d,r_2,\cdots,r_\ell)>0$ such that
\begin{align}\label{eq:multipart_recursion}
t_{K(1,r_2,\cdots,r_\ell)}(G)\geq 
(1-\eta)d^{r_2}t_{K(r_2+1,\cdots,r_\ell)}(G)
\end{align}
whenever $G$ is $(\rho,d)$-dense.
If the claim is true, then using \eqref{eq:multipart_recursion}
repeatedly to reduce the number of colour classes
and applying \eqref{eq:logconvex} to reduce the number of vertices in each class
yields
\begin{align*}
\frac{t_{K(r_1,r_2,\cdots,r_\ell)}(G)}{t_{K(r_1-1,r_2,\cdots,r_\ell)}(G)}
&\geq
\frac{(1-\eta)d^{r_2}t_{K(r_2+1,\cdots,r_\ell)}(G)}{t_{K(r_2,\cdots,r_\ell)}(G)}\\
&\geq\cdots\geq
(1-(\ell-1)\eta)d^{r_2+\cdots+r_\ell}\frac{t_{K(r_\ell+1)}(G)}{t_{K(r_\ell)}(G)}.
\end{align*}
Since both $K(r_{\ell}+1)$ and $K(r_\ell)$ consist of isolated vertices, 
taking $\eta=\delta/(\ell-1)$ is enough to conclude.

Thus, it remains to prove \eqref{eq:multipart_recursion}.
For brevity, let $H=K(r_2+1,r_3,r_4,\cdots,r_\ell)$ and let $h=|H|$.
Consider $\rho>0$ such that $\rho^{1+r_2}\leq \frac{1}{2}\eta d^{r_2+h(h-1)/2}$
and $t_{K_h}(G')\geq\frac{1}{2}d^{h(h-1)/2}$ whenever $G'$ is a sufficiently large $(\rho,d)$-dense graph.
Such $\rho$ exists by Theorem~\ref{thm:folklore}.
For a homomorphism $\phi$ from $K(r_3,r_4,\cdots,r_\ell)$ to $G$,
define $C_\phi$ to be the set of vertices $v$ such that 
adding $v$ to $\phi(K(r_3,\cdots,r_\ell))$ extends $\phi$ as a homomorphism 
from $K(1,r_3,r_4,\cdots,r_\ell)$ to $G$.
Denote by $\Phi$ the set of homomorphisms $\phi\in\Hom(H,G)$ such that $|C_\phi|\geq\rho |G|$.
Then we have
\begin{align}\label{eq:Cphi}
|\Hom(K(1,r_2,r_3,\cdots,r_\ell),G)|\geq
\sum_{\phi\in\Phi}|\Hom(K_{1,r_2},G[C_\phi])|
\geq d^{r_2}\sum_{\phi\in\Phi}|C_\phi|^{r_2+1},
\end{align}
where the last inequality follows from the fact 
$t_{K_{1,r}}(J)\geq t_{K_2}(J)^{r}$ for any graph $J$.
Note now that
\begin{align*}
\sum_{\phi\in\Phi}|C_\phi|^{r_2+1}
&=|\Hom(H,G)|
-\sum_{\phi\notin\Phi}|C_\phi|^{r_2+1}
\\
&\geq |\Hom(H,G)|-\rho^{1+r_2}|G|^{1+r_2}|\Hom(K(r_3,\cdots,r_\ell),G)|\\
&
\geq |\Hom(H,G)|-\rho^{1+r_2}|G|^{h}.
\end{align*}
Substituting this into \eqref{eq:Cphi} gives
\begin{align*}
|\Hom(K(1,r_2,r_3,\cdots,r_\ell),G)|\geq d^{r_2}|\Hom(H,G)|-\rho^{1+r_2}|G|^{|H|}.
\end{align*}
Therefore, by normalising both sides by $|G|^{|H|}$,
\begin{align*}
 t_{K(1,r_2,r_3,\cdots,r_\ell)}(G)\geq d^{r_2}t_H(G)-\rho^{1+r_2}\geq (1-\eta)d^{r_2}t_H(G),
\end{align*}
where the last inequality follows from 
\begin{align*}
\rho^{1+r_2}\leq \frac{1}{2}\eta d^{r_2+h(h-1)/2}\leq \eta d^{r_2}t_{K_h}(G)\leq \eta d^{r_2}t_{H}(G).\tag*{\qedhere} 
\end{align*}
\end{proof}
An immediate corollary is that we may compare $t_{K(r_1,r_2,\cdots,r_\ell)}(G)$ and 
$t_{K(s_1,s_2,\cdots,s_\ell)}(G)$
by repeatedly applying Lemma \ref{lem:multi}
whenever $r_i\geq s_i\geq 0$ for $i=1,2,\cdots,\ell$.
\begin{corollary}\label{cor:KrKs}
Suppose $\delta>0$ and $\ell,r_1,\cdots,r_\ell$, 
and $s_1,s_2,\cdots,s_\ell$ are 
positive integers with $r_i\geq s_i$, $i=1,2,\cdots,\ell$.
Let $r=e(K(r_1,r_2,\cdots,r_\ell))$ and $s=e(K(s_1,s_2,\cdots,s_\ell))$.
Then
there exists a positive $\rho=\rho(\delta,d,r_1,\cdots,r_\ell,s_1,\cdots,s_\ell)$ such that
\begin{align*}
 t_{K(r_1,\cdots,r_\ell)}(G)\geq (1-\delta)d^{r-s} t_{K(s_1,\cdots,s_\ell)}(G),
\end{align*}
whenever $G$ is a sufficiently large $(\rho,d)$-dense graph.\hfill$\square$
\end{corollary}

We now turn to the proof of Theorem~\ref{thm:multipartite}.

\begin{proof}[Proof of Theorem \ref{thm:multipartite}].
Let $K=K(r_1,r_2,\cdots,r_\ell)$ for brevity 
and let $(\FF,\TT)$ be a $K$-decomposition of a graph $H$.
By Theorem \ref{thm:folklore} we know that $\Hom(H[X\cap Y],G)$ is non-empty.
Hence, Theorem \ref{thm:tree_hom} gives
\begin{align}
		t_H(G)\geq \frac{t_{K}(G)^{|\FF|}}{\prod_{XY\in E(\TT)} t_{H[X\cap Y]}(G)}=\frac{t_K(G)^{e(\TT)+1}}{\prod_{XY\in E(\TT)} t_{H[X\cap Y]}(G)},
\end{align}
as $e(\TT)=|\FF|-1$.
Using the bound
$t_{K}(G)/t_{H[X\cap Y]}(G)\geq (1-\delta)d^{e(K)-e_H[X\cap Y]}$
from Corollary \ref{cor:KrKs}, 
we obtain
\begin{align*}
t_H(G)&\geq t_{K}(G)\prod_{XY\in E(\TT)} (1-\delta)d^{e(K)-e_H[X\cap Y]}\\
&\geq (1-\delta)d^{e(K)}\prod_{XY\in E(\TT)} (1-\delta)d^{e(K)-e_H[X\cap Y]}.
\end{align*}
By Lemma~\ref{lem:edge_counts}, $e(H) =|\FF|e(K)-\sum_{XY\in E(\TT)}e_H[X\cap Y]$,
and thus,
\begin{align*}
t_H(G)\geq (1-\delta)^{|\FF|}d^{e(H) }.
\end{align*}
Taking $\delta=\eta/|\FF|$ concludes the proof.
\end{proof}

\section{Counting $C_{2r+1}$-decomposable graphs}\label{sec:oddcycles}

We shall prove the remaining half of Theorem~\ref{thm:treelike} at the end of this section.
\begin{theorem}\label{thm:cycle}
Conjecture \ref{conj:KNRS} is true if $H$ is a $C_{2k+1}$-decomposable graph.
\end{theorem}

In proving Theorem \ref{thm:cycle}, we will again follow the same proof strategy.
In order to apply Theorem~\ref{thm:tree_hom},
the key will be to prove appropriate graph homomorphism inequalities 
between odd cycles and the paths they contain.
\begin{lemma}\label{lem:cyclepath}
Given $d,\delta>0$ and positive integers $\ell$ and $r$ with $\ell\leq 2r$,
there exists $\rho=\rho(\delta,d,r,\ell)$ such that
\begin{align}\label{eq:cyclepath}
	t_{C_{2r+1}}(G)\geq (1-\delta)d \cdot t_{P_\ell}(G)^{2r/\ell}
\end{align}
whenever $G$ is sufficiently large and $(\rho,d)$-dense.
\end{lemma}
There are two main ingredients in the proof of Lemma~\ref{lem:cyclepath}.
The first is Reiher's lemma, Lemma~\ref{lem:Reiher}, and the second is the log-convexity of path homomorphisms.
In fact, after obtaining the statement independently, we found that the log-convexity of paths has already been obtained by Erd\H{o}s and Simonovits~\cite{ES82}. We include a simple proof in the appendix for the sake of completeness.

\begin{lemma}\label{lem:paths}
For any graph $G$ and positive integers $\ell\leq 2r$, the following inequality holds: 
\begin{align}\label{eq:paths}
  t_{P_{2r}}(G)\geq t_{P_\ell}(G)^{2r/\ell}.
\end{align}
\end{lemma}

By Lemma \ref{lem:paths}, 
if we prove that there exists $\rho>0$ such that
\begin{align*}
t_{C_{2r+1}(G)}\geq (1-\delta)d\cdot t_{P_{2r}}(G)
\end{align*}
whenever $G$ is a sufficiently large $(\rho,d)$-dense graph,
then Lemma~\ref{lem:cyclepath} follows.
The proof of this inequality closely resembles that of Reiher \cite{Re14} proving Conjecture \ref{conj:KNRS} for odd cycles,
but, despite the similarity, the conclusion of  Lemma~\ref{lem:cyclepath} is slightly tighter than the standard application of Corollary~\ref{cor:Reiher}.
Hence, our proof only relies on Lemma~\ref{lem:Reiher}.
\begin{proof}[Proof of Lemma \ref{lem:cyclepath}] 
Let $|G|=n$ and let $q(v)$ be the normalised number of walks of length $r$ 
starting from $v$, i.e., we divide the number of walks by $n^{r-1}$.
Denote by $U:=\{u: q(u)> \rho n\}$ the set of vertices with large $q(u)$. Then
\begin{align*}
\frac{1}{n^{2r-2}}|\Hom(P_{2r},G)|
=\sum_{u\in U} q(u)^2 + \sum_{u\notin U} q(u)^2,
\end{align*}
and hence, 
\begin{align*}
\sum_{u\in U} q(u)^2 \geq |\Hom(P_{2r},G)|/n^{2r-2} - \rho^2 n^3.
\end{align*}
On the other hand, let $f_u(v)$ be the normalised number of walks of length $r$ from $u$ to $v$. 
Then by definition $q(u)=\sum_{v\in V(G)} f_u(v)$,
and 
$\sum_{v\in V(G)} f_u(v)> \rho n$ whenever $u$ is in $U$. 
For each $u \in U$, Lemma \ref{lem:Reiher} gives 
\begin{align*}
2\sum_{vw\in E(G)} f_u(v)f_u(w)
\geq d
\left(\sum_{v\in V(G)}f_u(v)\right)^2
-2n
\geq dq(u)^2 -2n.
\end{align*}
Summing this inequality over all $u\in U$ is 
at most the normalised number of homomorphisms
$|\Hom(C_{2r+1},G)|/n^{2r-2}$.
Hence, we have
\begin{align*}
\frac{1}{n^{2r-2}}|\Hom(C_{2r+1},G)|
\geq \frac{d}{n^{2r-2}}|\Hom(P_{2r},G)|-\rho^2 n^3-2n^2.
\end{align*}
Dividing both sides by $n^3$ gives
\begin{align*}
t_{C_{2r+1}}(G)\geq d\cdot t_{P_{2r}}(G)-\rho^2 -2/n.
\end{align*}
By taking $\rho=\sqrt{\delta d^{2r+1}}/4$ and $n>4/\delta d^{2r+1}$,
the inequality $t_{P_{2r}}(G)\ge d^{2r}$ finishes the proof.
\end{proof}

\medskip

\begin{proof}[Proof of Theorem \ref{thm:cycle}]
Let $(\FF,\TT)$ be a $C_{2r+1}$-decomposition of $H$ and set $\epsilon=\eta/|\FF|$.
By Reiher's theorem \cite{Re14} on odd cycles,\footnote{
Obviously, it also follows from Lemma \ref{lem:cyclepath} for the case $\ell=2r$,
which rephrases Reiher's argument.}
there exists $\rho=\rho(\delta,d,\epsilon)>0$ such that
\begin{align}\label{eq:Reiher}
 t_{C_{2r+1}}(G)\geq (1-\epsilon)d^{2r+1}
\end{align}
whenever $G$ is a sufficiently large $(\rho,d)$-dense graph.
Let $e_{XY}$ be the number of edges in $H[X\cap Y]$ for $XY\in E(\TT)$.
Each $H[X\cap Y]$, $XY\in E(\TT)$, is a vertex-disjoint union of paths,
and thus by Lemma~\ref{lem:paths} we obtain the upper bound
\begin{align*}
t_{H[X\cap Y]}(G)\leq t_{P_{2r}}(G)^{e_H[X\cap Y]/2r}.
\end{align*}
Combining this bound with Theorem \ref{thm:tree_hom}, we obtain
\begin{align*}
t_H(G)\geq \frac{t_{C_{2r+1}}(G)^{|\FF|}}{\prod_{XY\in E(\TT)} t_{H[X\cap Y]}(G)}
\geq 
\frac{t_{C_{2r+1}}(G)^{|\FF|}}{t_{P_{2r}}(G)^{\frac{1}{2r}\sum_{XY\in E(\TT)}e_H[X\cap Y]}}.
\end{align*}
The bound $t_{C_{2r+1}}(G)\geq (1-\epsilon)d\cdot t_{P_{2r}}(G)$ from Lemma~\ref{lem:cyclepath} now gives
\begin{align}\label{eq:compute}
t_H(G)\geq (1-\epsilon)^{|\FF|}d^{|\FF|}t_{P_{2r}}(G)^{|\FF|-\frac{1}{2r}\sum_{XY\in E(\TT)}e_H[X\cap Y]}.
\end{align}
By Lemma~\ref{lem:edge_counts}, 
\begin{align*}
e(H) =(2r+1)|\FF|-\sum_{XY\in E(\TT)}e_{XY}=|\FF|+2r\left(|\FF|-\frac{1}{2r}\sum_{XY\in E(\TT)}e_H[X\cap Y]\right),
\end{align*}
and thus, \eqref{eq:compute} yields
\begin{align*}
t_H(G)&\geq (1-|\FF|\epsilon)d^{|\FF|}t_{P_{2r}}(G)^{|\FF|-\frac{1}{2r}\sum_{XY\in E(\TT)}e_H[X\cap Y]}\\
&\geq (1-\eta)d^{|\FF|}t_{P_{2r}}(G)^{\frac{1}{2r}(e(H) -|\FF|)}
\geq (1-\eta)d^{e(H) }.\tag*{\qedhere} 
\end{align*}
\end{proof}

\section{Concluding remarks}
\textbf{On Li--Szegedy smoothness.}
The smoothness condition~\eqref{eq:smooth} and~\eqref{eq:large_entropy} can also be interpreted in terms of entropy. It is more convenient to use the term \emph{relative entropy} to the uniform vertex sampling, as in~\cite{Sz15}.
Namely, \eqref{eq:smooth} means that the following randomised algorithm guarantees `large' relative entropy: sampling $\phi$ with the distribution $p_\phi$ and sample a homomorphism in $\Hom(H,G)$ that extends $\phi$ uniformly at random. Indeed,~\eqref{eq:large_entropy} just means that $p_\phi$ itself has large enough entropy. We however used the language of logarithmic concavity instead of entropy, because it clarifies our key idea to control the rational exponents of $t_H(G;x,y)$.

If one takes the information-theoretic approach used in~\cite{CKLL15,KLL14,LSz12,Sz15} to prove Sidorenko's property of a graph $H$, it is often easy to find an induced subgraph $J$ that is smooth in $H$. For instance, if $H$ is a strongly tree-decomposable graph defined in~\cite{CKLL15}, every sub-decomposition defined in~\cite{CKLL18} induces a smooth subgraph in $H$. 
Moreover, in the proof of Theorem~\ref{thm:chord}, no special properties except the Sidorenko property of even cycles have been used. Thus, it is certainly possible to replace the even cycles by other graphs having Sidorenko's property.

It is hence possible to obtain more instances of Conjecture~\ref{conj:KNRS} using this method, once the smooth subgraph is `symmetric' enough to apply Lemma~\ref{lem:Reiher}. However, this still seems far from solving the full conjecture and we do not pursue this further.

\medskip

\noindent
\textbf{Smallest open case for Conjecture \ref{conj:KNRS}.} 
Using all the results so far and the theorem of Reiher \cite{Re14}
on odd cycles, one may check that Conjecture \ref{conj:KNRS} is true for $H$ with at most five vertices. However, we do not know how to handle the following special case for 6 vertices and leave it as an open problem.
\begin{question}
Let $H$ be the graph with 6 vertices and 8 edges,
where $V(H)=\mathbb{Z}_6$ and $E(H)=\{\{i,i+1\}:1\leq i\leq 6\}\cup\{\{1,5\},\{2,4\}\}$.
Does $H$ satisfy Conjecture~\ref{conj:KNRS}?
\end{question}

\medskip

\noindent
\textbf{Extending Theorem~\ref{thm:treelike} to chordal graphs.}
When using Theorem \ref{thm:tree_entropy},
the major caveat is 
how to find random variables $(X_{i,F})_{i\in F}$ that 
agree on the marginals $(X_{j,A})_{j\in A\cap B}$.
In fact, the symmetry condition in the definition of the $J$-decomposition 
is tailored to satisfy the marginal constraints.
However, for non-isomorphic graphs $H_1$ and $H_2$ containing the same induced subgraph $J$,
it is often hard to find distributions on $\Hom(H_1,G)$ and $\Hom(H_2,G)$
that agree on the natural projection to $\Hom(J,G)$.

For example, it is possible to generate a random copy of a tree in such a way that
the projection of the distribution onto a subtree agrees with the distribution 
generated by the same algorithm,
which Theorem~\ref{thm:subtree_smoothness} essentially implies.
This also leads to the definition of strongly tree-decomposable graphs used in \cite{CKLL15}.
Another example is Theorem~\ref{thm:edge_smooth},
whose consequence is that,
if $H$ has Sidorenko's property,
we may assume that 
the projection of a uniform random homomorphic copy of $H$ onto a single edge 
is again uniform.
Therefore, it is possible to glue graphs having Sidorenko's property on a single edge
while preserving the property.
Likewise, 
if it is possible to generate random copies of $K_r$ and $K_s$ 
that have the same marginals on $K_t$ for
every $t\leq \min(r,s)$
in locally dense graphs,
then it may be possible to prove Conjecture \ref{conj:KNRS} for all chordal graphs. We hence raise the following question:
\begin{question}
Do all chordal graphs satisfy Conjecture~\ref{conj:KNRS}?
\end{question}

\vspace{5mm}

\textbf{Acknowledgement.} I would like to thank David Conlon for bringing 
Conjecture \ref{conj:KNRS} to my attention and
for many helpful discussions. I would also like to thank Oliver Riordan, Mathias Schacht, and Sasha Sidorenko, who carefully read various versions of this paper and gave useful comments.

\bibliographystyle{abbrv}
\bibliography{references}

\appendix
\section{Proofs of auxiliary lemmas}
\begin{proof}[Proof of Lemma~\ref{lem:edge_counts}]
Let $uv\in E(H)$ and let $\FF(uv)$ be the collection of vertex sets $X\in \FF$ containing both $u$ and $v$.
Since $\FF(uv)=\FF(u)\cap\FF(v)$ and $\FF(uv)\neq\emptyset$ by (ii) of the definition of tree decompositions, 
$\FF(uv)$ induces a non-empty subtree $\TT_{uv}$ of $\TT$.
In \eqref{eq:edge_counts}, each $uv$ is counted $|\FF(uv)|$ times in $\sum_{X\in\FF}e_H[X]$
and $e(\TT_{uv})=|\FF(uv)|-1$ times in $\sum_{XY\in E(\TT)}e_H[X\cap Y]$.
Thus, every edge is counted once on both sides of \eqref{eq:edge_counts}, which proves the desired identity.
\end{proof}

\begin{proof}[Proof of Lemma~\ref{lem:kolmogorov}]
Suppose $X_1,X_2$, and $X_3$ take values in
finite sets $A,B,$ and $C$, respectively.
Note that $X_2'$ also takes values in $B$.
Let $Y_1,Y_2,$ and $Y_3$ be random variables defined by the joint distribution
\begin{align}\label{eq:conditional}
	\PP [Y_1=a, Y_2=b, Y_3=c]
	=\frac{\PP[X_1=a,X_2=b]~\PP[X_2'=b,X_3=c]}
	{\PP[X_2=b]}
\end{align}
for all possible values of $a,b$, and $c$.
This is well-defined, since $\PP[X_1=a,X_2=b]=0$ whenever $\PP[X_2=b]=0$.
We claim that $(Y_1,Y_2,Y_3)$ is a random vector with desired properties.
Summing \eqref{eq:conditional} over all $a\in A$ gives
\begin{align}\label{eq:cond_bc}
	\PP[ Y_2=b, Y_3=c]=\PP[ X_2'=b,X_3=c]
\end{align}
and the fact that $\PP[X_2=b]=\PP[X_2'=b]$ gives that by symmetry
\begin{align*}
	\PP[ Y_1=a,Y_2=b]=\PP[X_1=a, X_2=b].
\end{align*}
Moreover, by substituting \eqref{eq:cond_bc} into
\eqref{eq:conditional}, we obtain
\begin{align*}
	\PP [Y_1=a| Y_2=b, Y_3=c]&=\frac{\PP [Y_1=a, Y_2=b, Y_3=c]}{\PP[Y_2=b, Y_3=c]}\\
	&= \frac{\PP [X_1=a, X_2=b]}{\PP[X_2=b]}= \PP[Y_1=a|Y_2=b],
\end{align*}
which implies the conditional independence
of $Y_1$ and $Y_3$ given $Y_2$.
\end{proof}

\begin{proof}[Proof of Theorem \ref{thm:tree_entropy}]
	We use induction on $|\FF|$. The base case $|\FF|=1$ is trivially true.
	Fix a leaf $L$ of $\TT$ and let $\TT'$ be the tree $\TT\setminus L$ on $\FF':=\FF\setminus\{L\}$.
	By rearranging indices, we may assume that
	$L=\{t,t+1,\cdots,k\}$ for some $t\leq k$
	and that 
	$\FF'$ satisfies $\cup_{F\in\FF'}F=[\ell]$ for some $t\leq \ell\leq k$.
	Let $P$ be the neighbour of $L$ in $\TT$.
	Then $\{t,t+1,\cdots,\ell\}=L\cap P$. 
	By the inductive hypothesis, there is $\Y=(Y_1,Y_2,\cdots,Y_\ell)$ such that 
	$\Y_F:=(Y_i)_{i\in F}$ and $\X_F:=(X_{i;F})_{i\in F}$ are identically distributed
	for each $F\in\FF'$ and moreover,
	\begin{align}\label{eq:induction}
		\HH(\Y)
		=\sum_{F\in\FF'}\HH(\X_F)
		-\sum_{AB\in E(\TT')}\HH((X_{i;A})_{i\in A\cap B}).
	\end{align}
	Since $\Y_P$ and $\X_P$ are identically distributed and $\{t,t+1,\cdots,\ell\}\subseteq P$,
	 $(Y_t,Y_{t+1},\cdots,Y_\ell)$ and $(X_{t;L},X_{t+1;L},\cdots,X_{\ell;L})$ are identically distributed.
	 Thus, we may apply Lemma \ref{lem:kolmogorov} with 
	\begin{align*}
	& X_1=(Y_1,Y_2,\cdots,Y_{t-1}),~~ X_2=(Y_t,Y_{t+1},\cdots,Y_\ell),\\
	& X_2'=(X_{t;L},X_{t+1;L},\cdots,X_{\ell;L}), \text{ and } X_3=(X_{\ell+1;L},X_{\ell+2;L},\cdots,X_{k;L}).
	\end{align*}
	Then there exists
	$(Z_1,Z_2,\cdots,Z_k)$ such that
	$(Z_1,Z_2,\cdots,Z_{t-1})$ and $(Z_{\ell+1},\cdots,Z_{k})$ are conditionally independent given
	$(Z_t,Z_{t+1},\cdots,Z_{\ell})$,
	$(Z_1,\cdots,Z_\ell)$ and $\Y$ are identically distributed,
	and $(Z_t,Z_{t+1},\cdots,Z_k)$ and $\X_{L}$ are identically distributed.
	By conditional independence, we obtain
	\begin{align*}
		\HH(Z_1,Z_2,\cdots,Z_k)
		=\HH(\Y)+\HH(\X_{L})
		-\HH(Y_t,Y_{t+1},\cdots,Y_{\ell}).
	\end{align*}
	Now \eqref{eq:induction} and the fact that $\{t,t+1,\cdots,\ell\}=L\cap P$ 
	implies \eqref{eq:tree_entropy}.
\end{proof}

\begin{proof}[Proof of Lemma~\ref{lem:paths}]
We shall repeatedly use the inequality
\begin{align*}
	t_{P_{k+t}}(G)\leq t_{P_{2k}}(G)^{1/2}t_{P_{2t}}(G)^{1/2}
\end{align*}
that easily follows from the Cauchy--Schwarz inequality.
If \eqref{eq:paths} holds for all paths of even length, 
then by $t_{P_{2k+1}}(G)\leq t_{P_{2k}}(G)^{1/2}t_{P_{2k+2}}(G)^{1/2}$ we are done.
Thus, we may assume that our path is of even length $2\ell$.
We claim that the sequence $t_{P_{2k}}(G)$, $k=1,2,\cdots,r$, is log-convex.
Observe that the Cauchy--Schwarz inequality gives log-convexity for adjacent terms, i.e.,
\begin{align*}
	t_{P_{2k}}(G)\leq t_{P_{2(k+1)}}(G)^{1/2}t_{P_{2(k-1)}}(G)^{1/2}.
\end{align*}
We need the following folklore fact for convexity.
\begin{lemma}\label{lem:convex}
Let $a_0,a_1,\cdots,a_r$ be a real sequence such that, for every $i=1,2,\cdots,r-1$,
\begin{align*}
a_i\leq\frac{a_{i-1}+a_{i+1}}{2}.
\end{align*}
Then $a_k\leq\frac{1}{r}(ka_r+(r-k)a_0)$.
\end{lemma}

\begin{proof}
Note that $b_i:=a_{i}-a_{i-1}$ is increasing.
As $a_k-a_0=b_k+b_{k-1}+\cdots+b_1$,
\begin{align*}
r(a_k-a_0)&=k(b_1+\cdots +b_k)+(r-k)(b_1+\cdots +b_k)\\
&\leq k(b_1+\cdots +b_k)+k(b_{k+1}+b_{k+2}+\cdots+b_r)=k(a_r-a_0),
\end{align*}
which completes the proof.
\end{proof}
Now letting $a_k:=\log t_{P_{2k}}(G)$ in Lemma~\ref{lem:convex} gives $a_k\leq\frac{1}{r}(ka_r+(r-k)a_0)$.
Therefore, 
\begin{align*}
 t_{P_{2\ell}}(G)\leq t_{P_{0}}(G)^{\frac{r-\ell}{r}}
 t_{P_{2r}}(G)^{\frac{\ell}{r}}.
\end{align*}
Since $P_0$ is the single vertex graph, this completes the proof. 
\end{proof}

\end{document}